\documentclass[12pt,leqno]{article}

\usepackage{amsmath,amsthm,amsfonts,amssymb}

\theoremstyle{plain}
\newtheorem{theorem}{Theorem}[section]
\newtheorem{proposition}[theorem]{Proposition}
\newtheorem{lemma}[theorem]{Lemma}
\newtheorem{corollary}[theorem]{Corollary}

\theoremstyle{definition}

\newtheorem{example}[theorem]{Example}

\theoremstyle{remark}

\begin{document}

\title{Capacities on a Finite Lattice}

\author{
Motoya Machida \\
{\normalsize Department of Mathematics, Tennessee Technological University,
Cookeville, TN} \\
{\normalsize\texttt{mmachida@tntech.edu}} \\
}

\date{\today}

\maketitle

\begin{abstract}
In his influential work~\cite{choquet} Choquet systematically studied
capacities on Boolean algebras in a topological space,
and gave a probabilistic interpretation for completely monotone
(and completely alternating) capacities.
Beyond complete monotonicity
we can view a capacity
as a marginal condition for probability distribution
over the distributive lattice of dual order ideals.
In this paper we discuss a combinatorial approach
when capacities are defined over a finite lattice,
and investigate Fr\'{e}chet bounds given the marginal condition,
probabilistic interpretation of difference operators,
and stochastic inequalities with completely monotone capacities.

\medskip\par\noindent
{\em AMS\/} 2010 {\em subject classifications.\/}
Primary 60C05;
secondary 60E15.

\end{abstract}

\section{Introduction}
\setcounter{equation}{0}
\label{intro}

Let $L$ be a finite lattice with partial ordering~$\le$,
and let $\hat{0}$ and $\hat{1}$
denote the minimum and the maximum element of $L$.
A monotone function $\varphi$ on $L$ is called
a \emph{capacity} if $\varphi(\hat{0}) = 0$ and $\varphi(\hat{1}) = 1$.
Let $\mathcal{L}$ denote the collection of nonempty dual order ideals in $L$,
and let $\mathcal{X}$ be an $\mathcal{L}$-valued random variable
on some probability space $(\Omega, \mathbb{P})$,
distributed as $\mathbb{P}(\mathcal{X} = V) = f(V)$.
Assuming $\mathbb{P}(\hat{0}\in\mathcal{X}) = 0$,
we can construct a capacity $\varphi$ by
\begin{equation}\label{marginal}
\varphi(x) = \mathbb{P}(x \in \mathcal{X}),
\quad x \in L .
\end{equation}
From another viewpoint,
the collection of capacities on $L$ is a convex polytope,
any element of which can be represented as the convex combination
\begin{equation}\label{extreme.rep}
  \varphi(x) = \sum_{V\in\mathcal{L}} f(V) \chi_{V}(x),
\quad x \in L,
\end{equation}
where $\chi_{V}$ denotes an indicator function on $V$.
In the way of formulating \eqref{extreme.rep},
the weight $f(V)$ is viewed as
a \emph{probability mass function} (pmf)
for $\mathcal{X}$,
by which \eqref{extreme.rep} is deemed to be \eqref{marginal}.
This probabilistic interpretation of capacity
was first considered by Choquet~\cite{choquet}
and independently by Murofushi and Sugeno~\cite{murofushi}.
It should be noted, however, that
the choice of $f$ is not necessarily unique
(see Examples~\ref{cap.example} and~\ref{cm.example}).

Let $X$ be an $L$-valued random variable,
distributed as $\mathbb{P}(X = x) = f(x)$.
If $f(\hat{0}) = 0$ then
the \emph{cumulative distribution function} (cdf)
\begin{equation}\label{cdf}
  \varphi(x) = \sum_{y \le x} f(y),
  \quad x \in L,
\end{equation}
becomes a capacity,
also known as a \emph{belief function} in \cite{grabisch}.
The function $f$ in \eqref{cdf} is called the \emph{M\"{o}bius inverse} of $\varphi$.

For $a_1,a_2,\ldots \in L$,
we define the \emph{difference operator} $\nabla_{a_1}$ by
\begin{equation}\label{diff}
\nabla_{a_1} \varphi(x)
= \varphi(x) - {\varphi(x\wedge a_1)},
\quad x \in L,
\end{equation}
and the \emph{successive difference operator}
$\nabla_{a_1,\ldots,a_n}$ recursively by
\begin{equation}\label{succ.diff}
\nabla_{a_1,\ldots,a_n} \varphi
= \nabla_{a_n}(\nabla_{a_1,\ldots,a_{n-1}} \varphi),
\quad n = 2,3,\ldots.
\end{equation}
Then the monotonicity of $\varphi$
is characterized by $\nabla_a \varphi \ge 0$ for any $a \in L$.
Moreover, 
if $\nabla_{a_1,\ldots,a_n} \varphi \ge 0$
for any $a_1,\ldots,a_n \in L$ and for any $n \ge 1$
then $\varphi$ is called \emph{completely monotone}
(or monotone of order $\infty$; see \cite{choquet}).
The complete monotonicity of $\varphi$
is necessary and sufficient for the existence
of a (necessarily unique) pmf satisfying \eqref{cdf}.
This crucial observation
was made by Choquet~\cite{choquet}
for the class of compact sets in a topological space,
and it is now known as the Choquet theorem
which has been instrumental in the studies of random sets.
See \cite{molchanov} for a comprehensive review on random sets
on topological spaces.
This result in case of lattices
was due to Norberg~\cite{norberg.lattice}
who studied measures on continuous posets.

By equipping $\mathcal{L}$
with the order relation $U \preceq V$ by $U \supseteq V$,
we obtain the distributive lattice $\mathcal{L}$
which embeds $L$
as the subposet $\mathcal{L}_0 := \{\langle{a}\rangle^*: a \in L\}$
of principal dual order ideals.
Then we can introduce a completely monotone capacity $\Phi$ on $\mathcal{L}$,
and call it a \emph{completely monotone extension} of $\varphi$
if it satisfies the marginal condition
\begin{equation}\label{marginal.phi}
\varphi(x) = \Phi(\langle{x}\rangle^*),
\quad x \in L .
\end{equation}
The marginal condition \eqref{marginal.phi}
is equivalent to \eqref{extreme.rep},
and the pmf $f(V)$ can be obtained from the M\"{o}bius inversion of $\Phi$.
By the same token,
\eqref{marginal} is the marginal condition \eqref{marginal.phi}
when
$\Phi(U) = \mathbb{P}(\mathcal{X} \preceq U)$
is a cdf for $\mathcal{X}$.

In Section~\ref{difference} we investigate the properties of the M\"{o}bius inversion
by which the successive difference operators are fully characterized.
Particularly we can show the Choquet theorem for a finite lattice.
Consequently, we can represent the successive difference operator
\begin{equation}\label{nabla.rep}
\nabla_{a_1,\ldots,a_n} \varphi(x)
= \mathbb{P}(X \le x,\, X \not\le a_i \mbox{ for all $i=1,\ldots,n$})
\end{equation}
when $\varphi$ is completely monotone.

In Section~\ref{extension}
we consider the optimal bounds for $\Phi(U)$,
called \emph{Fr\'{e}chet bounds},
subject to the marginal condition \eqref{marginal.phi}.
We present a combinatorial approach to the
Fr\'{e}chet bounds,
and formulate the optimal lower bound
$\lambda(\varphi; a,b)$ for 
$\Phi(\langle{a,b}\rangle^*)$
at the dual order ideal $\langle{a,b}\rangle^*$ generated by a pair $\{a,b\}$ of $L$.
We can introduce a difference operator
by replacing $\varphi(a\wedge x)$
with $\lambda(\varphi; a,x)$ in \eqref{diff},
and call it ``$\lambda$-difference,''
denoted by $\Lambda_{a_1} \varphi$.
The resulting successive $\lambda$-difference operator
$\Lambda_{a_1,\ldots,a_n} \varphi$
parallels the characterization of $\nabla_{a_1,\ldots,a_n} \varphi$
via \eqref{nabla.rep}.
In Section~\ref{lambda.operators}
we can show that
there exists an $\mathcal{L}$-valued random variable $\mathcal{X}$
satisfying
\begin{equation}\label{Lambda.rep}
\Lambda_{a_1,\ldots,a_n} \varphi(x)
= \mathbb{P}(x\in\mathcal{X},\, a_i\not\in\mathcal{X}
    \mbox{ for all $i=1,\ldots,n$})
\end{equation}
given the marginal condition \eqref{marginal}.

In Section~\ref{prob}
we briefly discuss completely alternating capacities
and their probabilistic interpretation in terms of dual capacities.
Then we investigate a stochastic comparison
between $\varphi(x) = \mathbb{P}(x \in \mathcal{X})$
and $\psi(y) = \mathbb{P}(Y \le y)$,
and obtain a sufficient condition for $\mathbb{P}(Y \in \mathcal{X}) = 1$,
which is characterized by
the two types of difference operator introduced earlier.

Our notation of set operations is fairly standard.
The set difference $A\setminus B$ is defined by
$\{x\in A: x\not\in B\}$,
and the inclusion relation $A \subset B$
means that $A$ is a strictly smaller subset of $B$.

\section{Successive difference functionals}
\setcounter{equation}{0}
\label{difference}

By $R(L)$ we denote the space of real-valued functions on $L$.
In this section we consider \eqref{succ.diff} defined over $\varphi \in R(L)$.
The operator $\nabla_{a_1,\ldots,a_{n}}$
does not depend on the order of $a_i$'s.
It is also easy to see that
$\nabla_{a_1,\ldots,a_{n}}\varphi(x) = 0$
if $x \le a_i$ for some $i \le n$;
in particular,
if $a_n \le a_i$ for some $i \le n-1$
then
$\nabla_{a_1,\ldots,a_{n}}\varphi(x)
= \nabla_{a_1,\ldots,a_{n-1}}\varphi(x)
- \nabla_{a_1,\ldots,a_{n-1}}\varphi(x\wedge a_n)
= \nabla_{a_1,\ldots,a_{n-1}}\varphi(x)$.
Thus,
we can introduce the \emph{successive difference functional}
$\nabla_{A}^b \varphi = \nabla_{a_1,\ldots,a_n} \varphi(b)$
for a nonempty subset $A = \{a_1,\ldots,a_n\}$ of $L$ and $b \in L$.
We can expand it to
\begin{equation}\label{nabla.expansion}
\nabla_A^b \varphi =
  \sum_{A' \subseteq A} (-1)^{|A'|}
  \varphi({\textstyle\bigwedge}\! A'\wedge b) ,
\quad \varphi \in R(L),
\end{equation}
where
$$
\textstyle\bigwedge\! A' = \begin{cases}
  \hat{1} & \mbox{ if $A' = \emptyset$; } \\
  \bigwedge_{a\in A'} a & \mbox{ if $A' \neq \emptyset$, }
\end{cases}
$$
denotes the greatest lower bound of a subset $A'$ of $L$.
The M\"{o}bius inverse $f$ in \eqref{cdf} 
is uniquely determined by
\begin{equation}\label{mobius.inverse}
  f(x) = \sum_{y \le x} \varphi(y) \mu(y,x),
\end{equation}
where $\mu$ is called the \emph{M\"{o}bius function}.

Here we denote the half-open interval $\{x\in L: a \le x < b\}$
by $[a,b)$.
We say ``$b$ covers $a$''
if $a < b$ and there is no other element
between $a$ and $b$ (i.e., $[a,b) = \{a\}$),
and ``$A'$ dominates $A$''
if $A' \subseteq A$ and
for any $x \in A$ there exists some $y \in A'$ satisfying $x \le y$.
It is easy to see that
$\nabla_A^b =  \nabla_{A'}^b$
if $A'$ is a dominating subset of $A$.

The M\"{o}bius function over the lattice $L$
can be constructed via the ``cross-cut'' property of Lemma~\ref{cross-cut}.

\begin{lemma}[Corollary~3.9.4 of Stanley~\cite{stanley}]
\label{cross-cut}
Let $a < b$, and let $C \subseteq [a,b)$.
If $C$ dominates $[a,b)$ then
the M\"{o}bius function satisfies
$$
\mu(a,b) = \sum\limits_{k=1}^{|C|} (-1)^k N_k
$$
where $|C|$ denotes the number of elements in $C$,
and $N_k$ is the number of $k$-element subsets $C'$ of $C$
satisfying $\bigwedge\! C' = a$.
\end{lemma}

A nonempty subset of a poset is called \emph{antichain}
if any two distinct elements of the subset are incomparable;
a singleton $\{a\}$ is a trivial antichain.
Let $b \in L$ be fixed.
An $n$-element subset $A = \{a_1,\ldots,a_n\}$ of $L$
is said to be a \emph{$b$-meet antichain}
if $\{a_1\wedge b,\ldots,a_n\wedge b\}$ is an $n$-element antichain.
We call a singleton $\{a\}$ a trivial $b$-meet antichain
only when $b \not\le a$.

By $L_A^b := \{\bigwedge\! A' \wedge b: A' \subseteq A\}$
we denote the induced subposet of $L$.
Then $L_A^b$ is a lattice with the minimum $\bigwedge\! A \wedge b$,
and shares the same meet $\wedge$ with $L$.
If $A = \{a_1,\ldots,a_n\}$ is a $b$-meet antichain,
then the maximum $b$ of $L_A^b$
covers exactly $n$ elements
$a_1\wedge b,\ldots,a_n\wedge b$.

\begin{lemma}\label{nabla.mobius}
Let  $A$ be a $b$-meet antichain,
and let $\mu_A^b$ be the M\"{o}bius function
of the lattice $L_A^b$.
Then
$$
\nabla_A^b \varphi
= \sum_{x \in L_A^b} \varphi(x) \mu_A^b(x,b) .
$$
\end{lemma}

\begin{proof}
Let $x \in L_A^b\setminus\{b\}$ be fixed,
and let
$C_x = \{a\wedge b:  x \le a\wedge b,\,a \in A\}$
be a dominating subset of
$\{z \in L_A^b: x \le z < b\}$.
By Lemma~\ref{cross-cut}
we obtain
$$
\mu_A^b(x,b)
= \sum_{k=1}^{|C_x|} (-1)^k N_k
= \sum_{A' \subseteq A} (-1)^{|A'|}
    \chi_{\{{\bigwedge}\! A'\wedge b = x\}},
$$
where
$$
  \chi_{\{\cdots\}} = \begin{cases}
    1 & \mbox{ if $\{\cdots\}$ is true; } \\
    0 & \mbox{ if $\{\cdots\}$ is false, }
  \end{cases}
$$
is the indicator function for the statement $\{\cdots\}$.
Note that the right-hand expression of summation also produces
the value $\mu(b,b) = 1$ when $x = b$.
Thus,
we obtain
$$
\sum_{x \in L_A^b} \varphi(x) \mu_A^b(x,b)
= \sum_{x \in L_A^b} \varphi(x)
\sum_{A' \subseteq A} (-1)^{|A'|}
  \chi_{\{{\bigwedge}\! A'\wedge b = x\}},
$$
which is equal to \eqref{nabla.expansion}.
\end{proof}

For the next lemma we assume that
$b \not\le a $ for any $a \in A$.
Then we can find
a subset $\tilde{A} \subseteq A$
such that (i) $\tilde{A}$ is a $b$-meet antichain
and (ii) $a \in A$ implies $a\wedge b \le a'\wedge b$
for some $a' \in \tilde{A}$,
and call it a ``maximal $b$-meet antichain'' of $A$.
And we can reduce $\nabla_A^b$ to $\nabla_{\tilde{A}}^b$.

\begin{lemma}\label{nabla.reduce}
If $\tilde{A}$ is a maximal $b$-meet antichain of $A$
then $\nabla_A^b = \nabla_{\tilde{A}}^b$.
\end{lemma}

\begin{proof}
Assume $\tilde{A} \subset A$.
Let $a \in A \setminus \tilde{A}$
and $\tilde{a} \in \tilde{A}$
be such that $a\wedge b \le \tilde{a}\wedge b$.
Then we set
$A' = A\setminus\{a\}$
and $A'' = A'\setminus\{\tilde{a}\}$,
and obtain
$$
\nabla_A^b
= \nabla_{A''}^b
- \nabla_{A''}^{b \wedge \tilde{a}}
- \nabla_{A''}^{b \wedge a}
+ \nabla_{A''}^{b \wedge \tilde{a} \wedge a}
= \nabla_{A'}^b.
$$
We repeat further reduction, if necessary,
until $A' = \tilde{A}$.
\end{proof}

Theorem~\ref{pi.theorem} verifies \eqref{nabla.rep}
when the M\"{o}bius inverse $f$ represents
the pmf for an $L$-valued random variable $X$.

\begin{theorem}\label{pi.theorem}
The M\"{o}bius inverse $f$ of $\varphi$ satisfies
\begin{equation}\label{pi.formula}
\nabla_A^b \varphi = \sum_{x\in\pi_A^b} f(x)
\end{equation}
where
$$
\pi_A^b = \{x\in L: x \le b, x \not\le a
\mbox{\rm\ for all $a \in A$ }\}.
$$
\end{theorem}

\begin{proof}
If $b \le a$ for some $a \in A$ then
$\pi_A^b = \emptyset$,
for which we stipulate that the summation in~\eqref{pi.formula} vanishes.
Otherwise, we can find 
a maximal $b$-meet antichain $\tilde{A}$ of $A$.
It is easily observed that $\pi_{\tilde{A}}^b = \pi_A^b$;
thus,
it suffices to show \eqref{pi.formula} for $\tilde{A}$
by Lemma~\ref{nabla.reduce}.
Henceforth, we assume that $A$ is a $b$-meet antichain.

Here we can define the function $\tilde{f}$ on the lattice $L_A^b$
by setting
$$
  \tilde{f}(a) = \sum_{x \in \pi^a} f(x),
\quad a \in L_A^b,
$$
where
$\pi^a =
{\{x\in L: x \le a, x \not\le a'
\mbox{ whenever $a' < a$ in $L_A^b$ }\}}$.
Since $\{\pi^a\}_{a\in L_A^b}$ partitions $L$,
we obtain
$$
  \varphi(a) = \sum_{\text{$a' \le a$ in $L_A^b$}} \tilde{f}(a'),
\quad a \in L_A^b,
$$
which implies that
$\tilde{f}$ is the M\"{o}bius inverse of $\varphi$ over $L_A^b$.
In particular, we can show that $\tilde{f}(b) = \nabla_A^b \varphi$
by Lemma~\ref{nabla.mobius}.
Note that $\pi^b = \pi_A^b$.
Therefore, $\tilde{f}(b)$ is equal to the right-hand side of \eqref{pi.formula}.
\end{proof}

The following result is the immediate corollary which
implies the Choquet theorem for capacities on a finite lattice.

\begin{corollary}\label{cm}
Assume $\varphi(\hat{0}) \ge 0$.
The M\"{o}bius inverse $f$ of $\varphi$ is nonnegative
if and only if $\varphi$ is completely monotone.
\end{corollary}

\begin{proof}
Theorem~\ref{pi.theorem} clearly implies
the necessity of complete monotonicity.
Note that $f(\hat{0}) = \varphi(\hat{0}) \ge 0$.
For any $b > \hat{0}$
we can choose
the collection $A$ of all the elements covered by $b$,
and obtain $\pi_A^b = \{b\}$ and $\nabla_A^b \varphi = f(b)$
in Theorem~\ref{pi.theorem}.
Thus, the complete monotonicity of $\varphi$
is also sufficient.
\end{proof}

A subset $V$ of $L$ is called
an \emph{order ideal} (or a \emph{down-set})
if $x \le y$ and $y \in V$ imply $x \in V$.
By $\langle A \rangle$
we denote the order ideal
$\{x\in L: {x \le a} \mbox{ for some } {a \in A}\}$
generated by a subset $A$ of $L$.
Then there is the one-to-one correspondence between
antichains $A$ and nonempty order ideals $V$ via $V = \langle A \rangle$
(cf.~\cite{stanley}).
Furthermore,
we have $\nabla_A^b \equiv \nabla_V^b$
since $A$ dominates $V = \langle A \rangle$.

\begin{proposition}\label{f.domain}
Suppose that $V$ is an order ideal of $L$.
Then the M\"{o}bius inverse $f$ of $\varphi$
has the support
$\{x\in L: f(x) \neq 0\}$ on $V$
if and only if 
$\nabla_V^b \varphi = 0$ for every $b \not\in V$.
\end{proposition}

\begin{proof}
Let $A$ be the antichain corresponding to $V$
satisfying $V = \langle A \rangle$,
and let $\tilde{L}$ be the subposet of $L$
induced on the subset $L\setminus V$.
Then we can define the function $\tilde{\varphi}$ on $\tilde{L}$ by
setting
$$
  \tilde{\varphi}(b) = \sum_\text{$x \le b$ in $\tilde{L}$} f(x),
\quad b \in \tilde{L}.
$$
By restricting $f$ on $\tilde{L}$,
we can view $f$ as the M\"{o}bius inverse of $\tilde{\varphi}$.
By introducing the subset $\pi_A^b$ from Theorem~\ref{pi.theorem},
we can find that
$$
  \tilde{\varphi}(b) = \sum_{x \in \pi_A^b} f(x)
  = \nabla_A^b \varphi = \nabla_V^b \varphi .
$$
Hence,
$f \equiv 0$ on $\tilde{L}$ if and only if $\nabla_V^b \varphi = 0$ for all $b \in \tilde{L}$.
\end{proof}

\section{Completely monotone extensions}
\setcounter{equation}{0}
\label{extension}

A subset $U$ of $L$ is called
a \emph{dual order ideal}
(or an \emph{up-set}) if $x \in U$ and $x \le y$ imply $y \in U$.
By $\langle A \rangle^*$
we denote the up-set
$\{x\in L: {x \ge a} \mbox{ for some } {a \in A}\}$
generated by a subset $A$ of $L$;
thus, setting the one-to-one correspondence between
antichains $A$ and nonempty dual order ideals $U$
via $U = \langle A \rangle^*$.
We write simply $\langle a_1,\ldots,a_n \rangle^*$
if $A = \{a_1,\ldots,a_n\}$ is explicitly specified,
and particularly we call it \emph{principal}
when the up-set $\langle a \rangle^*$ is generated by a singleton $\{a\}$.
The collection $\mathcal{J}^*(L)$ of dual order ideals of $L$
is a distributive lattice ordered by inclusion (cf.~\cite{stanley}),
and so is the subposet of $\mathcal{J}^*(L)$
induced on the set of nonempty dual order ideals,
denoted by $\mathcal{L}$.
The poset $\mathcal{L}$
is poset-isomorphic to the distributive lattice of dual order ideals on the subposet
$L\setminus\{\hat{1}\}$.
In what follows we assume that
$\mathcal{L}$ is equipped with the reverse inclusion relation $\preceq$
so that $U \preceq V$ if $U \supseteq V$.

\begin{example}\label{lattice.example}
Let $L = \{\emptyset, 1, 2, 3, 12, 13, 23, 123\}$
be a three-element Boolean lattice ordered by inclusion,
where we express the subset $\{1, 2\}$ simply by ``$12$.''
Then the distributive lattice
\begin{align*}
\mathcal{L} = \{
&
  \langle{\emptyset}\rangle^*, \langle{1,2,3}\rangle^*,
  \langle{1,2}\rangle^*, \langle{1,3}\rangle^*, \langle{2,3}\rangle^*, 
  \langle{1,23}\rangle^*, \langle{2,13}\rangle^*, \langle{3,12}\rangle^*, 
  \langle{1}\rangle^*, \langle{2}\rangle^*, \langle{3}\rangle^*, \\
&
  \langle{12,13,23}\rangle^*,
  \langle{12,13}\rangle^*, \langle{12,23}\rangle^*, \langle{13,23}\rangle^*, 
  \langle{12}\rangle^*, \langle{13}\rangle^*, \langle{23}\rangle^*, \langle{123}\rangle^*
\} 
\end{align*}
has the minimum $\langle{\emptyset}\rangle^*$
and the maximum $\langle{123}\rangle^*$.
\end{example}

By $M_1(L)$ we denote the collection of nonnegative monotone functions on $L$,
and by $M_{\infty}(\mathcal{L})$ the collection of nonnegative completely monotone
functions on $\mathcal{L}$.
As $L$ is poset-isomorphic to the subposet $\mathcal{L}_0$
of $\mathcal{L}$ induced on the set of principal dual order ideals,
there is a natural projection $\Pi(\Phi) = \varphi$
via \eqref{marginal.phi}
from $\Phi\in M_{\infty}(\mathcal{L})$ to $\varphi \in M_1(L)$.
The map $\Pi$ is surjective, but not bijective unless $L$ is linearly ordered.
Proposition~\ref{surjective}
is given by Murofushi and Sugeno~\cite{murofushi}
who demonstrated a construction of \eqref{extreme.rep} by applying
a ``greedy method.''

\begin{proposition}\label{surjective}
The map $\Pi$ is surjective from $M_{\infty}(\mathcal{L})$ onto $M_1(L)$.
\end{proposition}

\begin{proof}
If $\varphi \equiv 0$ then $\Phi \equiv 0$ obviously satisfies
$\Phi(\Phi) = \varphi$.
Assume $\varphi\in M_1(L)$ with $\varphi(\hat{1}) > 0$.
Then we can consider the map $U(t) = \{a\in L: \varphi(a) > t\}$
from $[0,\varphi(\hat{1}))$ to $\mathcal{J}^*(L)$.
It is a step-wise decreasing map
$U(t) \equiv U(r_{i-1})$ for $t \in [r_{i-1}, r_i)$
with $0 = r_0 < r_1 < \cdots < r_m = \varphi(\hat{1})$.
Then we can assign $f(V) = r_i - r_{i-1} > 0$ if $V = U(r_{i-1})$ for some $i$;
otherwise, $f(V) = 0$.
Clearly the marginal condition \eqref{extreme.rep} holds,
and
\begin{equation}\label{phi.extension}
  \Phi(U) = \sum_{V \preceq U} f(V)
\end{equation}
determines
$\Phi\in M_{\infty}(\mathcal{L})$ as desired.
\end{proof}

\begin{example}\label{cap.example}
Let $L$ be the Boolean lattice from Example~\ref{lattice.example}.
Then
$$
\varphi_c(x) = \begin{cases}
  1   & \text{ if $x = 123$; } \\
  c   & \text{ if $x = 12$, $13$, or $23$; } \\
  0   & \text{otherwise, }
\end{cases}
$$
is a capacity on $L$ if $0 \le c \le 1$.
By the greedy method
we can construct a completely monotone extension
$$
\Phi_c(U) = \begin{cases}
  1 & \text{ if $U = \langle{123}\rangle^*$; }\\
  c & \text{ if $\langle{12,13,23}\rangle^* \preceq U \prec \langle{123}\rangle^*$; }\\
  0 & \text{ otherwise. }
\end{cases}
$$
\end{example}

If $\varphi$ is completely monotone
then the M\"{o}bius inverse $f$ of $\varphi$
can induce the \emph{M\"{o}bius extension} $\Phi$ via \eqref{phi.extension}
by setting $f(\langle{x}\rangle^*) = f(a)$ for $x\in L$
and $f \equiv 0$ on $\mathcal{L}\setminus\mathcal{L}_0$.
The converse is also true:
If the M\"{o}bius inverse $f$ of
a completely monotone extension $\Phi$ of $\varphi$
has the support $\{U\in\mathcal{L}: f(U) \neq 0\}$ in $\mathcal{L}_0$
then $\varphi$ is completely monotone,
uniquely formulated by \eqref{cdf} with $f(x) = f(\langle{x}\rangle^*)$.

\begin{example}\label{cm.example}
In Example~\ref{cap.example}
we can find $\varphi_{1/3} \in M_{\infty}(L)$.
Then the M\"{o}bius inverse
$$
f(V) = \begin{cases}
  1/3 & \text{ if $V = \langle{12}\rangle^*$, $\langle{13}\rangle^*$,
                      or $\langle{23}\rangle^*$; } \\
  0     & \text{ otherwise, }
\end{cases}
$$
determines the M\"{o}bius extension $\Phi$ of $\varphi_{1/3}$.
\end{example}

The M\"{o}bius extension can be characterized by its values
at dual order ideals of the form $\langle{a,b}\rangle^*$.

\begin{proposition}\label{cm-prop}
$\Phi$ is the M\"{o}bius extension of $\varphi$
if and only if
\begin{equation}\label{ab-formula}
\Phi(\langle{a,b}\rangle^*) = \varphi(a \wedge b)
\mbox{ for every pair $\{a,b\}$. }
\end{equation}
\end{proposition}

\begin{proof}
Let $f$ be the M\"{o}bius inverse  of $\Phi$.
Then we can observe that
$$
\Phi(\langle{a,b}\rangle^*) = \varphi(a \wedge b)
 + \sum \{f(U): \mbox{$a,b \in U$ and $a \wedge b \not\in U$}\} .
$$
Hence, $\Phi$ is the M\"{o}bius extension of $\varphi$
and $f$ is supported by $\mathcal{L}_0$
if and only if it satisfies \eqref{ab-formula}.
\end{proof}

\subsection{Fr\'{e}chet bounds}
\label{lower.bounds}

Kellerer~\cite{kellerer} and
R\"{u}schendorf~\cite{ruschendorf} investigated
the optimal bounds analogous to the classical Fr\'{e}chet bounds
systematically for various marginal problems.
Let $R(\mathcal{L})$ be the space of real-valued functions on $\mathcal{L}$.
Given $\Phi\in M_{\infty}(\mathcal{L})$
we can formulate the nonnegative linear functional
$$
  \Phi(g) = \sum_{V\in\mathcal{L}} f(V) g(V),
\quad g \in R(\mathcal{L}),
$$
where $f$ is the M\"{o}bius inverse of $\Phi$.
Assuming $\varphi\in M_1(L)$,
we can define the Fr\'{e}chet bound
\begin{equation}\label{b.bound}
  B_{\varphi}(g) = \min\{\Phi(g): \Pi(\Phi) = \varphi\}
\end{equation}
for any $g \in R(\mathcal{L})$.
Duality follows from the relationship
between primal and dual problem
of linear programming,
but it is also viewed as a straightforward application
of the Hahn-Banach theorem (cf. Kellerer~\cite{kellerer}).

\begin{theorem}\label{duality.theorem}
The dual problem
\begin{equation}\label{dual.bound}
S^{\varphi}(g) = \max\left\{
  \sum_{x\in L} r_x \varphi(x):
  \sum_{x\in V} r_x \le g(V), V\in\mathcal{L}
\right\} .
\end{equation}
satisfies
$B_{\varphi}(g) = S^{\varphi}(g)$
for any $g \in R(\mathcal{L})$.
\end{theorem}

\begin{proof}
We can introduce
a function of the form
\begin{equation}\label{r.form}
r(V) = \sum_{x\in L} r_x \chi_{\{x\in V\}},
\quad V\in\mathcal{L}
\end{equation}
so that the inequality constraints in \eqref{dual.bound}
are simply stated as $r \le g$.
Suppose that $\Phi_0\in\Pi^{-1}(\varphi)$ attains $B_{\varphi}(g)$,
and that $r_0$ of the form \eqref{r.form} satisfies $r_0 \le g$
and attains $S^{\varphi}(g)$.
Then we obtain
$S^{\varphi}(g) = \Phi_0(r_0) \le \Phi_0(g) = B_{\varphi}(g)$.
Thus, $S^{\varphi}(g)$ is a lower bound for $B_{\varphi}(g)$,
and the equality holds if $g$ is in a form of \eqref{r.form}.

Now let $g \in R(\mathcal{L})$ be fixed.
Since
$S^{\varphi}$ is sublinear, satisfying
$S^{\varphi}(g_1 + g_2) \ge S^{\varphi}(g_1) + S^{\varphi}(g_2)$,
by the Hahn-Banach theorem
we can find a linear functional $\Psi$
such that $S^{\varphi}(h) \le \Psi(h)$
for any $h \in R(\mathcal{L})$,
in which the equality holds if $h$
is in the form of \eqref{r.form} or $h = g$.
Then $\Psi$ is a nonnegative linear functional
corresponding to $\Psi \in M_{\infty}(\mathcal{L})$,
and it satisfies $\Pi(\Psi) = \varphi$.
Hence, we have shown that
$B_{\varphi}(g) \le \Psi(g) = S^{\varphi}(g)$,
which completes the proof.
\end{proof}

Let $U\in\mathcal{L}$,
and let $g_U(V) = \chi_{\{V\preceq U\}}$.
Then we have $\Phi(U) = \Phi(g_U)$,
and accordingly we simply write $B_{\varphi}(U)$
for $B_{\varphi}(g_U)$ in \eqref{b.bound}.
In the rest of this subsection
we investigate the Fr\'{e}chet bound $B_{\varphi}(U)$.

\begin{proposition}\label{cm-prop.2}
If $\varphi \in M_{\infty}(L)$ then
$B_{\varphi}(U)$ is the M\"{o}bius extension of $\varphi$.
\end{proposition}

\begin{proof}
For each $U\in\mathcal{L}$,
we can express $U = \langle{A}\rangle^*$ with antichain $A$,
and observe that
$\varphi(\textstyle\bigwedge\! A)
\le B_{\varphi}(\langle{A}\rangle^*)$.
Let $\Phi$ be the M\"{o}bius extension of $\varphi$.
Then we can find
$\Phi(\langle{A}\rangle^*)
= \varphi(\textstyle\bigwedge\! A)$,
and therefore,
$\Phi(\langle{A}\rangle^*)
= B_{\varphi}(\langle{A}\rangle^*)$.
\end{proof}

\begin{example}
In general, the Fr\'{e}chet bound $B_{\varphi}(U)$
may not be a completely monotone extension of $\varphi$.
Continuing from Example~\ref{cap.example},
we can find that
$$
B_{\varphi_{2/3}}(U) = \begin{cases}
  1 & \text{ if $U = \langle{123}\rangle^*$; }\\
  2/3 & \text{ if $U = \langle{12}\rangle^*$, $\langle{13}\rangle^*$,
                      or $\langle{23}\rangle^*$; } \\
  1/3 & \text{ if $U = \langle{12,13}\rangle^*$,
                      $\langle{12,23}\rangle^*$, or $\langle{13,23}\rangle^*$;} \\
  0     & \text{ otherwise, }
\end{cases}
$$
is a completely monotone extension of $\varphi_{2/3}$
even though $\varphi_{2/3} \not\in M_{\infty}(L)$.
Whereas,
$$
B_{\varphi_{1/2}}(U) = \begin{cases}
1     & \mbox{ if $U = \langle{123}\rangle^*$; } \\
1/2 & \mbox{ if $U = \langle{12}\rangle^*, \langle{13}\rangle^*,$
                       or $\langle{23}\rangle^*$; } \\
0     & \mbox{ otherwise, }
\end{cases}
$$
is not completely monotone.
\end{example}

By $\mathcal{T}$ we denote the class of connected acyclic graphs (i.e., trees) with
vertex set on $L$.
The vertex set of a tree $G$ is also denoted by $G$,
and the edge set $E(G)$ is a collection of pairs $\{a,b\}$ in $G$.
Then we can associate a tree $G$ with $\varphi$
by setting
$$
  \varphi(G) = \sum_{a\in G} \varphi(a)
  - \sum_{\{a,b\}\in E(G)} \varphi(a\vee b).
$$

Let $a \in G$ be fixed.
Then we can introduce the unique rooted tree on $G$ as follows:
For $x,y \in G$,
$x$ is a descendant of $y$
(and $y$ is an ancestor of $x$)
if the path from $x$ to $a$ in $G$ contains the path from $y$ to $a$,
and $a$ becomes the root of the tree.
The rooted tree is a directed graph (digraph)
in which the ordered pair $(x,y)$
represents the edge with $y$ being the parent of $x$ (i.e., the immediate ancestor of $x$).
By $E(G;a)$ we denote the edge set of the rooted tree with the root $a$.
By defining
$$
  \varphi(G;a) = 
  \sum_{(x,y)\in E(G;a)} [\varphi(x\vee y) - \varphi(x)],
$$
we can formulate $\varphi(G)$ equivalently by
\begin{equation}\label{varphi.G}
  \varphi(G) = \varphi(a) - \varphi(G;a).
\end{equation}
Observe that $\varphi(G;a) \ge 0$,
and therefore, that $\varphi(G) \le \varphi(a)$.
Moreover, we can obtain the following result
as an immediate application of Theorem~\ref{duality.theorem}.

\begin{lemma}\label{ub.lemma}
$\varphi(G) \le B_{\varphi}(\langle{G}\rangle^*)$
for any $G \in\mathcal{T}$.
\end{lemma}

In the proof of Lemma~\ref{ub.lemma}
it is convenient to define a graph restricted on a down-set:
For a tree $G\in\mathcal{T}$
and a down-set $D$,
we will define the subgraph $G|_{D}$
by setting $G|_{D} := G\cap D$
and $E(G|_{D}) := \{\{a,b\}\in E(G): a\vee b \in D\}$.

\begin{proof}[Proof of Lemma~\ref{ub.lemma}]
Let
$g(V) = \chi_{\{V\preceq\langle{G}\rangle^*\}}$ and
$$
r(V) =\sum_{a\in G} \chi_{\{a\in V\}}
- \sum_{\{a,b\}\in E(G)} \chi_{\{a\vee b\in V\}}
$$
for $V \in \mathcal{L}$.
Note that $r$ is in the form of \eqref{r.form}.
Since $|G| = |E(G)| + 1$,
we can observe that
$r(V) = g(V) = 1$ if $V\preceq\langle{G}\rangle^*$.
Suppose that $V\not\preceq\langle{G}\rangle^*$.
Then the down-set $D = L\setminus V$ contains at least one vertex of $G$.
If the subgraph $G|_D$ has $k$ connected components,
we can find that $r(V) = 1 - k \le 0$.
Thus, we obtain $r \le g$, and therefore, $\varphi(G) \le S^{\varphi}(g)$.
The proof is complete by Theorem~\ref{duality.theorem}.
\end{proof}

In what follows we say ``a path $H$ from $a$ to $b$,''
or simply write $H = (a,\ldots,b)$
when $H\in\mathcal{T}$ and $a$ and $b$ are the only leaves in $H$
(i.e., the two opposite ends of the path).
By Lemma~\ref{ub.lemma}
we have
$\varphi(H) \le B_{\varphi}(\langle{H}\rangle^*)
\le B_{\varphi}(\langle{a,b}\rangle^*)$ if $H = (a,\ldots,b)$.
In Proposition~\ref{attain.ub} we shall see that
\begin{equation}\label{lambda}
  \lambda(\varphi;a,b)
  := \max\{\varphi(H): \mbox{$H$ is a path from $a$ to $b$}\}
\end{equation}
is optimal.
It is easy to observe that
$\lambda(\varphi;a,b) \ge \varphi(a \wedge b)$;
in particular, $\lambda(\varphi;a,b) \ge 0$
if $\varphi \ge 0$.
Furthermore,
we can view $\lambda(\varphi;a,x)$
as a function of $x$,
and obtain the monotonicity property.

\begin{lemma}\label{lambda.monotone}
If $\varphi\in M_1(L)$ then
so does $\lambda(\varphi;a,\cdot)$.
\end{lemma}

\begin{proof}
Let $H_1 = (a,\ldots,x)$ be a path satisfying
$\varphi(H_1) = \lambda(\varphi;a,x)$,
and let $x < y$.
Without loss of generality
we can assume that $y \not\in H_1$.
Then we can add the edge $\{x,y\}$ to $H_1$,
and obtain the path $\tilde{H}_1 = (a,\ldots,x,y)$.
Since
$\varphi(H_1) = \varphi(\tilde{H}_1) \le \lambda(\varphi;a,y)$,
we have shown that $\lambda(\varphi;a,\cdot)$ is monotone.
\end{proof}

For any $a\in L$
we can introduce the \emph{$\lambda$-difference operator}
$\Lambda_a$ by
\begin{equation}\label{Lambda}
\Lambda_a \varphi(x)
= \varphi(x) - \lambda(\varphi;a,x),
\quad x \in L.
\end{equation}
By \eqref{varphi.G} and \eqref{lambda}
we can easily see that \eqref{Lambda} is expressed equivalently by
\begin{equation}\label{Lambda.H}
\Lambda_a\varphi(x)
= \min\{\varphi(H;x):  \mbox{$H$ is a path from $a$ to $x$}\}.
\end{equation}
Clearly $\Lambda_a\varphi \ge 0$ if $\varphi$ is monotone,
and it also possesses the monotonicity property.

\begin{lemma}\label{Lambda.monotone}
If $\varphi\in M_1(L)$ then
so does $\Lambda_a\varphi$.
\end{lemma}

\begin{proof}
By \eqref{Lambda.H}
we can find a path $H_2 = (a,\ldots,y)$
such that $\varphi(H_2;y) = \Lambda_a\varphi(y)$.
Let $x < y$.
If $x\in H_2$ then we can construct
the path $\tilde{H}_2 = (a,\ldots,x)$
by deleting all the edges from $x$ to $y$ in $H_2$,
and obtain $\varphi(H_2;y) \ge \varphi(\tilde{H}_2;x)$.
Otherwise, we can add the edge $\{y,x\}$ to $H_2$,
and the resulting path $\tilde{H}_2 = (a,\ldots,y,x)$
satisfies $\varphi(H_2;y) = \varphi(\tilde{H}_2;x)$.
In either case we can show that
$\varphi(H_2;y) \ge \varphi(\tilde{H}_2;x) \ge \Lambda_a\varphi(x)$.
Therefore, $\Lambda_a\varphi$ is monotone.
\end{proof}

Now we can prove the optimality of \eqref{lambda}.

\begin{proposition}\label{attain.ub}
$\lambda(\varphi;a,b) = B_{\varphi}(\langle{a,b}\rangle^*)$
for every pair $\{a,b\}$ of $L$.
\end{proposition}

\begin{proof}
For a fixed $a \in L$,
we can decompose
$\varphi(\cdot)
= \lambda(\varphi;a,\cdot) + \Lambda_a\varphi(\cdot)$,
in which
$\lambda(\varphi;a,\cdot), \Lambda_a\varphi(\cdot) \in M_1(L)$
by Lemma~\ref{lambda.monotone} and~\ref{Lambda.monotone}.
Thus, we can find
completely monotone extensions $\Phi_1$ and $\Phi_2$
of $\lambda(\varphi;a,\cdot)$
and $\Lambda_a\varphi(\cdot)$ respectively,
and construct $\Phi = \Phi_1 + \Phi_2$
so that $\Pi(\Phi) = \varphi$.
Observe that
$$
\Phi_2(\langle{a,x}\rangle^*) \le \Phi_2(\langle{a}\rangle^*)
= \Lambda_a\varphi(a) = 0.
$$
and therefore, that
$$
  \Phi(\langle{a,x}\rangle^*) = \Phi_1(\langle{a,x}\rangle^*)
  \le \Phi_1(\langle{x}\rangle^*) = \lambda(\varphi;a,x).
$$
Since $\lambda(\varphi;a,x) \le B_{\varphi}(\langle{a,x}\rangle^*)$
by Lemma~\ref{ub.lemma},
$\lambda(\varphi;a,x)$ attains $B_{\varphi}(\langle{a,x}\rangle^*)$.
\end{proof}

\subsection{Successive $\lambda$-difference operators}
\label{lambda.operators}

Given a sequence $a_1,a_2,\ldots$ from $L$,
we can define the \emph{successive $\lambda$-difference operator}
recursively by
\begin{equation}\label{Lambda.suc}
\Lambda_{a_1,\ldots,a_n}\varphi
= \Lambda_{a_n}(\Lambda_{a_1,\ldots,a_{n-1}}\varphi),
\quad n = 2,3,\ldots .
\end{equation}
The operator \eqref{Lambda} maps from $M_1(L)$ to itself,
and so does the operator \eqref{Lambda.suc}.
Unlike the operator \eqref{succ.diff},
the definition of \eqref{Lambda.suc} depends on the order of $a_i$'s,
as illustrated in the following example.

\begin{example}\label{lattice.4}
Let $L = \{\emptyset,1,2,3,4,
12,13,14,23,24,34,
123,124,134,234,
1234\}$
be a four-element Boolean lattice,
and let
\begin{equation}\label{varphi.4}
\varphi(x) = \begin{cases}
1                  & \mbox{ if $x = 1234$; } \\
1/2 & \mbox{ if $x = 123,124$ or $234$; } \\
1/3 & \mbox{ if $x = 134, 13$ or $23$; } \\
1/6 & \mbox{ if $x = 12$ or $34$; } \\
0                  & \mbox{ otherwise. }
\end{cases}
\end{equation}
Then we have
$\Lambda_{12,34}\varphi(234) = \frac{1}{3}$
and
$\Lambda_{34,12}\varphi(234) = \frac{1}{6}$.
If $x \neq 234$ then we obtain
$$
\Lambda_{12,34}\varphi(x) = \Lambda_{34,12}\varphi(x)
= \begin{cases}
2/3 & \mbox{ if $x = 1234$; } \\
1/3 & \mbox{ if $x = 124$; } \\
1/6 & \mbox{ if $x = 13, 23, 123$ or $134$; } \\
0                  & \mbox{ otherwise. }
\end{cases}
$$
\end{example}

We call a path $(a_1,\ldots,a_n)$ \emph{monotone}
if $i < j$ whenever $a_i < a_j$.
As the following lemma suggests,
we only need to consider a monotone path $(a_1,\ldots,a_n)$
for the operator $\Lambda_{a_1,\ldots,a_n}$.

\begin{lemma}\label{Lambda.contract}
If $a_n \le a_i$ for some $i \le n-1$
then
$\Lambda_{a_1,\ldots,a_{n}}\varphi
= \Lambda_{a_1,\ldots,a_{n-1}}\varphi$
for every $\varphi\in M_1(L)$.
\end{lemma}

\begin{proof}
Let $\varphi_{n-1} = \Lambda_{a_1,\ldots,a_{n-1}}\varphi$.
Since $a_n \le a_i$,
$\varphi_{n-1}(a_n) \le \Lambda_{a_1,\ldots,a_{i}}\varphi(a_n) = 0$.
Thus, we can find that
the path $H_0 = (a_n,a_n\wedge x, x)$ attains the minimum
$\Lambda_{a_{n}}\varphi_{n-1}(x) = \varphi_{n-1}(x)$.
\end{proof}

Here we set $\varphi_0 = \varphi \in M_1(L)$
and $\varphi_i = \Lambda_{a_i}\varphi_{i-1}$
recursively for $i=1,\ldots,n$.
Then we can express $\varphi_k$ by
\begin{equation}\label{Lambda.expansion}
\varphi_k(\cdot) =
\sum_{i=k}^{n-1} \lambda(\varphi_i; a_{i+1},\cdot)
+ \varphi_n(\cdot),
\quad k = 0,\ldots,n-1.
\end{equation}
By choosing
$\Psi_i\in\Pi^{-1}(\lambda(\varphi_i,a_{i+1},\cdot))$
for $i=0,\ldots,n-1$,
and $\Psi_n\in\Pi^{-1}(\varphi_n)$,
we can construct
\begin{equation}\label{Phi.cons}
  \Phi = \sum_{i=0}^n \Psi_i.
\end{equation}
Comparing \eqref{Phi.cons} with
\eqref{Lambda.expansion} at $k=0$,
we can easily observe that $\Pi(\Phi) = \varphi$.
Theorem~\ref{Lambda.cons}
characterizes $\Lambda_{a_1,\ldots,a_k}\varphi$;
in particular, when $\varphi$ is a capacity
there exists an $\mathcal{L}$-valued random variable $\mathcal{X}$
satisfying \eqref{marginal} and \eqref{Lambda.rep}.

\begin{theorem}\label{Lambda.cons}
Let $(a_1,\ldots,a_n)$ be a monotone path, and let
\begin{equation}\label{Lambda.cons.g}
\pi^x_{a_1,\ldots,a_k}(V) = \begin{cases}
  1 & \mbox{ if $x\in V$,  $a_i \not\in V$ for all $i=1,\ldots,k$; } \\
  0 & \mbox{ otherwise, }
\end{cases}
\end{equation}
be an indicator function on $\mathcal{L}$.
Then \eqref{Phi.cons} satisfies
\begin{equation}\label{Lambda.cons.eq}
\Lambda_{a_1,\ldots,a_k}\varphi(x)
= \Phi\left(\pi^x_{a_1,\ldots,a_k}\right),
\quad x \in L ,
\end{equation}
for $k=1,\ldots,n$.
\end{theorem}

\begin{proof}
Let $f_i$ be the M\"{o}bius inverse of $\Psi_i$ for $i=0,\ldots,n$.
For each $i = 0,\ldots,{n-1}$,
note that
$\lambda(\varphi_i; a_{i+1}, a_{i+1}) =
\lambda(\varphi_i; a_{i+1}, \hat{1}) = \varphi_i(a_{i+1})$,
and therefore, that ${f_i(V) > 0}$ implies $a_{i+1}\in V$.
In particular,
we find
$\Psi_i\left(\pi^x_{a_1,\ldots,a_k}\right) = 0$ for $i=0,\ldots,k-1$.
For any $i = 1,\ldots,k$
we can observe that
$\lambda(\varphi_j; a_{j+1}, a_i) = 0$ for $j = k,\ldots,n-1$
and that $\varphi_n(a_i) = 0$;
thus, $f_j(V) = 0$ for $j=k,\ldots,n$ if $a_i \in V$
for some $i = 1,\ldots,k$.
Hence, we obtain
$\Psi_j\left(\pi^x_{a_1,\ldots,a_k}\right)
= \Psi_j\left(\langle{x}\rangle^*\right)$
for $j = k,\ldots,n$.
Together we can establish
$$
\Phi\left(\pi^x_{a_1,\ldots,a_k}\right)
= \sum_{j=k}^n \Psi_j\left(\langle{x}\rangle^*\right)
= \sum_{j=k}^{n-1} \lambda(\varphi_j; a_{j+1}, x) + \varphi_n(x)
= \varphi_k(x)
$$
where we can apply \eqref{Lambda.expansion}
for the last equality.
\end{proof}

By Theorem~\ref{pi.theorem}
we can find that
the operator $\nabla_{a_1,\ldots,a_n}$ maps $M_{\infty}(L)$ to itself.
Furthermore, it coincides with the operator $\Lambda_{a_1,\ldots,a_n}$
on $M_{\infty}(L)$.

\begin{lemma}\label{Lambda.cm}
$\Lambda_{a_1,\ldots,a_n}\varphi
= \nabla_{a_1,\ldots,a_n} \varphi$
for $\varphi\in M_{\infty}(L)$.
\end{lemma}

\begin{proof}
We prove it by induction.
Suppose that $\varphi_{n-1} = \Lambda_{a_1,\ldots,a_{n-1}}\varphi
= \nabla_{a_1,\ldots,a_{n-1}} \varphi$.
Since $\varphi_{n-1}\in M_{\infty}(L)$,
by Propositions~\ref{cm-prop} and~\ref{cm-prop.2}  we obtain 
$\lambda(\varphi_{n-1}; a_n, x) = \varphi_{n-1}(a_n\wedge x)$,
and therefore,
$\Lambda_{a_n}\varphi_{n-1} = \nabla_{a_n}\varphi_{n-1}$.
\end{proof}

A monotone path $(a_1,\ldots,a_n)$ is viewed as a linear extension of $L$
if $\{a_1,\ldots,a_n\}$ is the entire set $L$.
As a corollary to Lemma~\ref{Lambda.cm}
we can find the uniqueness of \eqref{Phi.cons} when $\varphi\in M_{\infty}(L)$.

\begin{corollary}\label{Lambda.cm.cor}
If $(a_1,\ldots,a_n)$ is a linear extension of $L$
and $\varphi\in M_{\infty}(L)$ then
\eqref{Phi.cons} is the M\"{o}bius extension of $\varphi$.
\end{corollary}

\begin{proof}
Observe that $\varphi_n \equiv 0$,
and that \eqref{Phi.cons} becomes
$\Phi = \sum_{i=0}^{n-1} \Psi_i$.
As we have shown in the proof of Lemma~\ref{Lambda.cm},
we have $\lambda(\varphi_i, a_{i+1},x) = \varphi_i(a_{i+1}\wedge x)$
for $i=0,\ldots,n-1$.
Since $(a_1,\ldots,a_n)$ is a linear extension of $L$,
we can see that $\varphi_i(x) = 0$ if $x \le a_i$,
and therefore, that
$\varphi_{i}(a_{i+1}\wedge x)
= \varphi_{i}(a_{i+1})\chi_{\langle{a_{i+1}}\rangle^*}(x)$;
thus, $\lambda(\varphi_i, a_{i+1},\cdot)$
has the unique completely monotone extension
$\Phi_i(V) = \varphi_{i}(a_{i+1})\chi_{\{\langle{a_{i+1}}\rangle^*\preceq V\}}$.
Hence, $\Phi$ must be the M\"{o}bius extension of $\varphi$.
\end{proof}

\section{Probabilistic interpretation}
\setcounter{equation}{0}
\label{prob}

By $C_1(L)$ we denote the collection of capacities on $L$,
and by $C_{\infty}(\mathcal{L})$ the collection of
completely monotone capacities on $\mathcal{L}$.
Proposition~\ref{surjective} indicates that
the projection $\Pi$ is surjective from $C_{\infty}(\mathcal{L})$ onto $C_1(L)$.
In view of \eqref{cdf} and Corollary~\ref{cm}
we can view any completely monotone capacity as a cdf.
In this section we consider
lattice-valued random variables
on some probability space $(\Omega, \mathbb{P})$,
and investigate their properties
which facilitate a probabilistic interpretation of capacities.

\subsection{Dual capacities}
\label{dual}

By $L^*$ we denote the dual lattice of $L$,
in which $\hat{0}$ and $\hat{1}$ respectively
become the maximum and the minimum.
Here we can introduce the successive difference operator $\nabla_{b_1,\ldots,b_n}$ on $L^*$,
and call it the \emph{dual successive difference}, denoted by $\Delta_{b_1,\ldots,b_n}$.
For any sequence $b_1, b_2,\ldots$ of $L$,
it can be constructed with the {dual difference operator}
$$
\Delta_{b_1} \varphi(x)
= \varphi(x) - \varphi(x\vee b_1),
$$
and recursively by
$$
\Delta_{b_1,\ldots,b_n} \varphi
= \Delta_{b_n}(\Delta_{b_1,\ldots,b_{n-1}} \varphi),
\quad n = 2,3,\ldots .
$$
Then a capacity $\varphi$ is called
\emph{completely alternating}
if $\Delta_{b_1,\ldots,b_n} \varphi \le 0$
for any sequence $b_1,\ldots,b_n$ of $L$ and for any $n \ge 1$.
Given $\varphi\in C_1(L)$,
we can introduce $\varphi^*\in C_1(L^*)$
by setting $\varphi^*(x) = 1 - \varphi(x)$ for $x \in L^*$,
and call it the \emph{dual capacity} of $\varphi$.
The duality immediately implies that
$\varphi$ is completely alternating
if and only if
$\varphi^*$ is completely monotone on $L^*$.

Let $\mathcal{X}$ be an $\mathcal{L}$-valued random variable.
Then $\varphi(x) = \mathbb{P}(x \in \mathcal{X})$
is a capacity if and only if 
$\mathbb{P}(\mathcal{X} = \langle{\hat{0}}\rangle^*) = 0$,
in which
$\Phi(U) = \mathbb{P}(\mathcal{X} \preceq U)$
is a completely monotone extension of $\varphi$.
By $\mathcal{L}^*$ we denote the distributive lattice of nonempty order ideals in $L$
(i.e., the distributive lattice of nonempty dual order ideals in $L^*$)
equipped with the reverse inclusion order $\preceq$
(i.e., $D \preceq E$ on $\mathcal{L}^*$ if $D \supseteq E$).
Assume $\mathbb{P}(\mathcal{X} = \langle{\hat{0}}\rangle^*) = 0$.
We can view the complement $\mathcal{X}^c = L\setminus \mathcal{X}$ as an
$\mathcal{L}^*$-valued random variable,
and define the \emph{dual extension}
$$
\Phi^*(D) = \mathbb{P}(\mathcal{X}^c \preceq D),
\quad D\in\mathcal{L}^* .
$$
It is easy to observe that
$$
\mathbb{P}(x \in \mathcal{X}^c)
= \mathbb{P}(x \not\in \mathcal{X})
= 1 - \mathbb{P}(x \in \mathcal{X})
= 1 - \varphi(x) = \varphi^*(x) ,
$$
and therefore, that
$\Phi^*$ is a completely monotone extension of $\varphi^*$.

Suppose that $\varphi$ is completely alternating
and $\Phi^*(D) = \mathbb{P}(\mathcal{X}^c \preceq D)$
is the M\"{o}bius extension of $\varphi^*$.
Then
the \emph{dual M\"{o}bius extension}
$\Phi(U) = \mathbb{P}(\mathcal{X} \preceq U)$
has the M\"{o}bius inverse $f$ supported by the collection
$$
\{U \in \mathcal{L}:
\mbox{$L\setminus U$ is a principal order ideal\/}
\}.
$$

\begin{proposition}\label{ca-prop}
A capacity $\varphi$ is completely alternating
and $\Phi$ is the dual M\"{o}bius extension of $\varphi$
if and only if
\begin{equation}\label{abab-formula}
\Phi(\langle{a,b}\rangle^*)
= \varphi(a) + \varphi(b) - \varphi(a \vee b)
\mbox{ for every pair $\{a,b\}$. }
\end{equation}
\end{proposition}

\begin{proof}
Let $\mathcal{X}$ be an $\mathcal{L}$-valued random variable
realizing $\varphi(x) = \mathbb{P}(x \in \mathcal{X})$.
Then $\mathcal{X}^c = L \setminus \mathcal{X}$ realizes
its dual $\varphi^*(x) = \mathbb{P}(x \in \mathcal{X}^c)$.
Thus, we obtain
\begin{align*}
\Phi(\langle{a,b}\rangle^*)
& = \mathbb{P}(\mathcal{X} \preceq \langle{a,b}\rangle^*)
   = \mathbb{P}(a \not\in \mathcal{X}^c, b \not\in \mathcal{X}^c)
\\
& = 1 - \mathbb{P}(a \in \mathcal{X}^c) - \mathbb{P}(a \in \mathcal{X}^c)
        + \mathbb{P}(a,b \in \mathcal{X}^c)
\\
& = \varphi(a) + \varphi(b) - \varphi(a \vee b)
+ \mathbb{P}(a,b \in \mathcal{X}^c, a \vee b \not\in \mathcal{X}^c) .
\end{align*}
If $\Phi$ is the dual M\"{o}bius extension of $\varphi$
then $\mathbb{P}(a,b \in \mathcal{X}^c, a \vee b \not\in \mathcal{X}^c) = 0$.
Conversely if \eqref{abab-formula} holds then
$\Phi^*$ must be the M\"{o}bius extension of $\varphi^*$.
\end{proof}

Since $\varphi(H) = \varphi(a) + \varphi(b) - \varphi(a \vee b)$
for a path $H = (a,b)$,
the dual M\"{o}bius extension
$\Phi(\langle{a,b}\rangle^*)$
in \eqref{abab-formula}
attains the Fr\'{e}chet bound $B_{\varphi}(\langle{a,b}\rangle^*)$.

\subsection{Stochastic inequalities}
\label{comparisons}

When $\varphi\in C_{\infty}(L)$ is a cdf
for $L$-valued random variable $X$,
by Theorem~\ref{pi.theorem}
we can show that
\begin{equation}\label{pi.rv}
\mathbb{P}(X\not\in\langle{A}\rangle) = \nabla_A^{\hat{1}}\varphi,
\quad A \subseteq L .
\end{equation}
Suppose that $(X,Y)$ is a pair of 
$L$-valued random variables.
We can construct such a pair satisfying
$\mathbb{P}(X \le Y) = 1$
if and only if
\begin{equation}\label{kko.condition}
\mathbb{P}(X \in U) \le \mathbb{P}(Y \in U)
\quad\mbox{ for every $U \in \mathcal{L}$, }
\end{equation}
given the marginal conditions $\varphi(x) = \mathbb{P}(X \le x)$
and $\psi(y) = \mathbb{P}(Y \le y)$.
By applying \eqref{pi.rv},
we can immediately observe that
\eqref{kko.condition} can be equivalently stated by
\begin{equation}\label{norberg.condition}
\nabla_{a_1,\ldots,a_k}\varphi(\hat{1})
\le \nabla_{a_1,\ldots,a_k}\psi(\hat{1})
\quad\mbox{ for every antichain $\{a_1,\ldots,a_k\}$ in $L$. }
\end{equation}
The stochastic inequality \eqref{kko.condition}
first appeared in the paper by 
Kamae, Krengel, and O'Brien \cite{kko},
and \eqref{norberg.condition} was shown by
Norberg \cite{norberg} in the context of random sets.

Let $\mathcal{X}$ be an $\mathcal{L}$-valued random variable,
and let $Y$ be an $L$-valued random variable.
In this subsection
we will investigate when we can construct a pair $(\mathcal{X}, Y)$
of random variables jointly so that $\mathbb{P}(Y \in \mathcal{X}) = 1$
given the marginal conditions
\begin{equation}\label{marginal.XY}
\varphi(x) = \mathbb{P}(x \in \mathcal{X})
\mbox{ and }
\psi(y) = \mathbb{P}(Y \le y),
\quad x,y \in L .
\end{equation}
The joint cdf $\Gamma(V,y) = \mathbb{P}(\mathcal{X}\preceq V,\, Y \le y)$
is a completely monotone capacity
on the direct product lattice $\mathcal{L}\times L$.
Given a joint cdf $\Gamma$,
we can introduce the expectation $E[w(\mathcal{X},Y)]$,
also denoted by $\Gamma(w)$, for $w \in R(\mathcal{L}\times L)$.
Then we can define the Fr\'{e}chet bound
$$
B^{(\varphi, \psi)}(w) = \max\{\Gamma(w)
  \mbox{ subject to \eqref{marginal.XY}}\},
\quad w \in R(\mathcal{L}\times L) .
$$
Similarly by $\psi(h)$ we denote the expectation $E[h(Y)]$ for $h\in R(L)$.
Recall the dual problem $S^{\varphi}(g)$ in Theorem~\ref{duality.theorem}.
In Theorem~\ref{joint.duality} we will show that
the Fr\'{e}chet bound $B^{(\varphi, \psi)}(w)$ has
the dual problem
\begin{equation}\label{joint.dual.S}
S_{(\varphi,\psi)}(w) = \min\{\psi(h) - S^{\varphi}(g)
\mbox{ subject to (\ref{dual.condition})} \}
\end{equation}
with the inequality constraint
\begin{equation}\label{dual.condition}
w(V,y) \le h(y) - g(V),
\quad (V,y) \in \mathcal{L}\times L,
\end{equation}
for $(g,h)\in R(\mathcal{L})\times R(L)$.

\begin{theorem}\label{joint.duality}
$B^{(\varphi,\psi)}(w) = S_{(\varphi,\psi)}(w)$
for any $w \in R(\mathcal{L}\times L)$.
\end{theorem}

\begin{proof}
Suppose that a joint cdf $\Gamma$ for $(\mathcal{X},Y)$
attains $B^{(\varphi,\psi)}(w)$,
that $(g,h)$ attains $S_{(\varphi,\psi)}(w)$,
and that $r$ is of the form \eqref{r.form}
satisfying $r \le g$ and $S^{\varphi}(g) = E[r(\mathcal{X})]$.
Then we can observe that
$$
S_{(\varphi,\psi)}(w) = \psi(h) - S^{\varphi}(g)
= E[h(Y)] - E[r(\mathcal{X})]
\ge E[w(\mathcal{X},Y)] = B^{(\varphi,\psi)}(w),
$$
and that the equality holds if
$w(V,y) = h(y) - r(V)$, $(V,y)\in\mathcal{L}\times L$.
Since $S_{(\varphi,\psi)}(w_1+w_2) \le S_{(\varphi,\psi)}(w_1) + S_{(\varphi,\psi)}(w_2)$,
we can apply the Hahn-Banach theorem analogous to the proof of
Theorem~\ref{duality.theorem},
and conclude that $B^{(\varphi,\psi)}(w) = S_{(\varphi,\psi)}(w)$.
\end{proof}

In what follows we consider 
the indicator function
$w_1(V,y) := \chi_{\{y\in V\}}$
for the dual problem \eqref{joint.dual.S}.
Starting with $g \in R(\mathcal{L})$,
we can construct the two monotone functions $h'$ and $g'$ by
\begin{align}
\label{h.dash}
h'(y) & = \max_{V\in\mathcal{L}}\left(w_1(V,y) + g(V)\right),
\quad y \in L; \\
\label{g.dash}
g'(V) & = \min_{y\in L}\left(h'(y) - w_1(V,y)\right),
\quad V \in \mathcal{L}.
\end{align}
By \eqref{g.dash}
we can see that \eqref{dual.condition} holds for $w_1$, $h'$, and $g'$.
Observe that
if $w_1$, $h$, and $g$ satisfy \eqref{dual.condition}
then $h \ge h'$ and $g' \ge g$ so that
$\psi(h) - S^{\varphi}(g) \ge \psi(h') - S^{\varphi}(g')$.
Thus, it suffices for us to consider the case when $g$ and $h$ are monotone.
Moreover, without loss of generality we can set
$g(\langle{\hat{1}}\rangle^*) = 0$ in addition to 
the constraint \eqref{dual.condition}.
Starting with a monotone function $g$ with $g(\langle{\hat{1}}\rangle^*) = 0$,
we obtain $0 \le h'(y) \le 1$ in \eqref{h.dash},
and $g'(V) = \min_{y\in V} h'(y) - 1$ in \eqref{g.dash}.
Therefore, we can further simplify \eqref{joint.dual.S} into
\begin{equation}\label{S.w.one}
S_{(\varphi,\psi)}(w_1) = \min\{\psi(h) - S^{\varphi}(\tilde{h})
   \mbox{ subject to (\ref{h.tilde})}\}
+ 1
\end{equation}
with the constraint
\begin{equation}\label{h.tilde}
  0 \le h(y) \le 1
  \mbox{ and }
  \tilde{h}(V) = \min_{y\in V} h(y),
\quad (V,y) \in \mathcal{L}\times L,
\end{equation}
for any monotone function $h \in R(L)$.

\begin{theorem}\label{comp.theorem}
If
\begin{equation}\label{comp.condition}
\Lambda_{a_1,\ldots,a_{k}}\varphi(\hat{1})
\le \nabla_{a_1,\ldots,a_{k}}\psi(\hat{1})
\quad
\mbox{ for every monotone path $(a_1,\ldots,a_k)$, }
\end{equation}
then there exists a joint cdf $\Gamma$ for $(\mathcal{X},Y)$
satisfying $\mathbb{P}(Y\in \mathcal{X}) = 1$ given the marginal conditions
\eqref{marginal.XY}.
\end{theorem}

\begin{proof}
Suppose that $h$ is a monotone function, and that \eqref{h.tilde} holds for
$(h,\tilde{h})$.
Then we can find a linear extension $(a_1,\ldots,a_N)$ of $L$
such that $h(a_i) \le h(a_j)$ whenever $i < j$.
By Theorem~\ref{Lambda.cons} we can construct
$\Phi\in\Pi^{-1}(\varphi)$ so that
\eqref{Lambda.cons.eq} holds for
the indicator function $\pi^x_{a_1,\ldots,a_k}$
with any choice of $k = 1,\ldots,N$.
For each $0 \le t < h(\hat{1})$,
note that there is some $k\le N-1$ such that
\begin{align*}
A(t) & := \{y\in L: h(y) > t\} = \{a_{k+1},\ldots,a_N\}; \\
\mathcal{A}(t) & := \{V\in\mathcal{L}: \tilde{h}(V) > t\}
= \{V\in\mathcal{L}: a_i\not\in V,\, i=1,\ldots,k\}.
\end{align*}
By applying Theorems~\ref{pi.theorem} and~\ref{Lambda.cons},
we can establish
\begin{align*}
& \psi(h) - S^{\varphi}(\tilde{h})
\ge \psi(h) - \Phi(\tilde{h})
= \int_{0}^{h(\hat{1})} \psi\left(\chi_{A(t)}\right)\,dt
- \int_{0}^{h(\hat{1})} \Phi\left(\chi_{\mathcal{A}(t)}\right)\,dt
\\ & \hspace{0.2in}
\ge
\min_{1\le k\le N-1}\left[
  \psi\left(\chi_{\{a_{k+1},\ldots,a_N\}}\right)
  - \Phi\left(\pi^{\hat{1}}_{a_1,\ldots,a_k}\right)
\right]
\\ & \hspace{0.2in}
= \min_{1\le k\le N-1}\left[
  \nabla_{a_1,\ldots,a_{k}}\psi(\hat{1})
  - \Lambda_{a_1,\ldots,a_{k}}\varphi(\hat{1})
\right]
\ge 0 .
\end{align*}
By Theorem~\ref{joint.duality} and \eqref{S.w.one}
we obtain
$B^{(\varphi,\psi)}(w_1) = S_{(\varphi,\psi)}(w_1) \ge 1$,
which implies the existence of a joint cdf $\Gamma$
satisfying $\mathbb{P}(Y\in \mathcal{X}) = 1$.
\end{proof}

\begin{example}
A stochastically comparable pair $(\mathcal{X},Y)$
does not necessarily satisfy \eqref{comp.condition}.
Let $L$ be the Boolean lattice from Example~\ref{lattice.4},
and let
$$
\Gamma(V,y) = \begin{cases}
1/6 & \mbox{ if $(V,y) = (\langle{12}\rangle^*,12),
(\langle{13,23,34}\rangle^*,34),(\langle{13,23}\rangle^*,234)$ } \\
       & \mbox{ or $(\langle{234}\rangle^*,234)$; } \\
1/3 & \mbox{ if $(V,y) = (\langle{124}\rangle^*,124)$; } \\
0     & \mbox{ otherwise, }
\end{cases}
$$
be a joint cdf for $(\mathcal{X},Y)$.
Then it satisfies $\mathbb{P}(Y\in \mathcal{X}) = 1$,
and $\varphi(x) = \mathbb{P}(x \in \mathcal{X})$
is equal to \eqref{varphi.4}.
By applying the result of Example~\ref{lattice.4},
we can calculate
$\Lambda_{34,12,234}\varphi(\hat{1})
= \Lambda_{34,12}\varphi(\hat{1})
- \Lambda_{34,12}\varphi(234)
= 1/2$.
Since
$
\nabla_{34,12,234}\psi(\hat{1})
= \mathbb{P}(Y \not\in \langle{12,234}\rangle)
= 1/3$,
it does not satisfy
\eqref{comp.condition}
for the monotone path
$(34,12,234)$.
\end{example}

\end{document}